\documentclass[a4paper,12pt]{article}

\usepackage[english]{babel}
\usepackage[latin1]{inputenc}
\usepackage{amsmath}
\usepackage{amsthm}
\usepackage{amsfonts}
\usepackage{amssymb, latexsym, graphics, showidx}
\usepackage{color}

\numberwithin{equation}{section}

\newtheorem{thm}{Theorem}[section]

\newtheorem{rem}[thm]{Remark}

\newtheorem{lem}[thm]{Lemma}

\renewcommand{\dim}{\begin{proof}}

\newcommand{\finedim}{\end{proof}}
\newcommand{\R}{\mathbb{R}}
\newcommand{\Tt}{\mathbb{T}}

\newcommand{\al}{\alpha}

\newcommand{\ep}{\epsilon}

\newcommand{\Hh}{\overline{H}}
\newcommand{\di}{\text{div}}
\newcommand{\intor}{\int_{\Tt^N}}

\renewcommand{\div}{\operatorname{div}}

\newcommand{\Rr}{{\mathbb{R}}}

\newcommand{\beq}{\begin{equation}}
\newcommand{\eeq}{\end{equation}}
\newcommand{\beqs}{\begin{equation*}}
\newcommand{\eeqs}{\end{equation*}}
\newcommand{\beqa}{\begin{eqnarray}}
\newcommand{\eeqa}{\end{eqnarray}}
\newcommand{\beqas}{\begin{eqnarray*}}
\newcommand{\eeqas}{\end{eqnarray*}}

\title{Obstacle Mean-Field Game Problem}
\author{Diogo
  A. Gomes\footnote{
King Abdullah University of Science and Technology (KAUST), CSMSE Division,Thuwal 23955-6900,  Saudi Arabia,
 and KAUST SRI, Uncertainty Quantification Center in Computational Science and Engineering. e-mail: diogo.gomes@kaust.edu.sa}, 
  Stefania Patrizi\footnote{Weierstrass Institute for Applied Analysis and Stochastics. e-mail: stefaniapatrizi@yahoo.it}
}

\date{\today} 

\begin{document}

\maketitle

\begin{abstract}
In this paper, we introduce and study a first-order mean-field game obstacle problem. 
We examine the case of local dependence on the measure under assumptions that include
both the logarithmic case and power-like nonlinearities. 
Since the obstacle operator is not differentiable, 
the equations for first-order mean field game problems have to be discussed carefully. 
Hence, we begin by considering a penalized problem.
We prove this problem admits a unique solution
satisfying uniform bounds. These bounds serve to pass to the limit in the penalized problem 
and to characterize the limiting equations. 
Finally, we prove uniqueness of solutions. 
\end{abstract}

\thanks{
D. Gomes was partially supported by KAUST baseline and start-up funds and 
KAUST SRI, Uncertainty Quantification Center in Computational Science and Engineering. }

\section{Introduction}

The mean-field game framework \cite{C1, C2, ll1, ll2, ll3, ll4}
is a class of methods that model the behavior of 
large populations of rational agents under a non-cooperative dynamic behavior. 
This research area has applications ranging from economics to engineering, 
as discussed in the recent surveys \cite{llg2, cardaliaguet, GS}, the additional
references therein, and the lectures by P. L. Lions in Coll\'ege de France \cite{LCDF}.   

In this paper, we investigate first-order mean-field game obstacle problems in the stationary periodic  setting. 
To our knowledge, in the context of mean-field games, 
these problems were not studied previously. 
Before describing the problem, we start by recalling the original stationary mean-field game problem from \cite{ll1},
as well as the obstacle problem for Hamilton-Jacobi (H-J) equations \cite{L}.  
 
Let $\Tt^N$ be the $N$-dimensional torus identified when convenient with $[0, 1]^N$.  Consider a continuous function,
 $H:\Rr^N\times \Tt^N\to \Rr$, the Hamiltonian, and a continuous increasing function, $g: \Rr^+_0\to \Rr$. 
In \cite{ll1}, the authors consider the stationary mean-field game system
\begin{equation}
\label{smfg}
\begin{cases}
H(Du,x)=g(\theta)+\Hh\\
\div(D_pH \theta)=0,    
\end{cases}
\end{equation}
where the unknowns are a function $u:\Tt^N\to \Rr$, a probability measure identified with its density $\theta:\Tt^N\to \Rr$ and a constant $\Hh$.
The second equation is the adjoint of the linearization of the first equation in the variable $u$. This system \eqref{smfg} has the canonical structure
of a mean-field game problem: a nonlinear elliptic or parabolic nonlinear partial differential equation (PDE) coupled with a PDE given by the adjoint of its linearization. 

The existence of weak solutions for \eqref{smfg} was considered in \cite{ll1}. 
In 
\cite{E2} mean-field games are not mentioned explicitly, however, 
the results there yield the existence of smooth solutions of \eqref{smfg}
for $g(\theta)=\ln \theta$. 
The second order case, was also studied in 
 \cite{GM} (see also \cite{GIMY}), \cite{GPM1}, and \cite{GPatVrt}). Stationary mean-field games with congestion were 
considered in \cite{GMit}.
The
time-dependent problem was addressed for parabolic mean-field games in \cite{ll2}, \cite{CLLP}, \cite{porretta}, \cite{GPM2}, \cite{GPM3}, \cite{GPim1}, \cite{GPim2}, 
and in \cite{Cd1}, and \cite{Cd2} for first-order
mean-field games. 

The first-order obstacle
problem arises in optimal stopping  (see \cite{L}, \cite{MR921827}, \cite{MR957658}, \cite{Bardi} and the references therein).
In the periodic setting, a model problem is the following:
let  $\psi:\Tt^N\to \Rr$, and $H:\Tt^N\times \Rr^N \to \Rr$ be continuous functions. 
The obstacle problem is defined by
\begin{equation}
\label{obstacleHJ}
\max\{H(Du,x), u-\psi(x)\}=0,
\end{equation}
where  $u:\Tt^N\to \Rr$ is a bounded continuous function.

The linearization of the obstacle
operator 
is not well-defined since the left-hand side of \eqref{obstacleHJ} may fail to be differentiable. 
Thus,
  it is not clear what should be the corresponding mean-field model. 
One of the 
contributions of this paper is the characterization of the appropriate analog to \eqref{smfg} for 
 obstacle problems. This is achieved by applying the penalization method. 
This is a standard technique employed in many related problems, e.g. \cite{L}.
In the classical obstacle problem, 
to do so, one considers a family 
of smooth functions, $\beta_\epsilon:\Rr\to \Rr_0^+$, which vanish identically in $\Rr_0^-$
and satisfy $\beta_\epsilon(z)=\frac{z-\epsilon}{\epsilon}$ for $z>\epsilon$. 
Then, obstacle problem is approximated by the equation
\begin{equation}
\label{approxobs}
H(Du_\ep,x)+\beta_\epsilon(u_\ep-\psi)=\epsilon \Delta u_\ep.
\end{equation}
This equation admits viscosity solutions that satisfy uniform Lipschitz bounds. 
By sending $\epsilon\to 0$, one obtains a solution to \eqref{obstacleHJ}.

Thanks to \cite{L}, for every $\epsilon > 0$ 
there exists a smooth solution $u^\varepsilon$ to \eqref{approxobs}.
It is also well known that, up to subsequences, $u^{\epsilon}$
converges uniformly to a viscosity solution $u$ of \eqref{obstacleHJ}.
The rate of convergence of this approximation was investigated using the nonlinear adjoint method in \cite{CGT2}.

We then are 
 led naturally to the approximate mean-field obstacle problem
\begin{equation}
\label{amfgobs}
\begin{cases}
H(Du_\epsilon,x)+\beta_\epsilon(u_\epsilon-\psi)=g(\theta_\epsilon)\\
-\div(D_pH(Du_\epsilon,x)\theta_\epsilon)+\beta_\epsilon'(u_\epsilon-\psi) \theta_\epsilon=\gamma(x).
\end{cases}
\end{equation}
The additional term $\gamma$ in the right-hand side of \eqref{amfgobs} arises
for the following reason: 
the mean-field obstacle problem models a population of agents trying to move optimally up to a certain stopping time at which they  switch to the obstacle 
(the term $\beta_\epsilon' \theta_\epsilon$ is the flow of agents switching to the obstacle).  
Without a source term introducing new agents in the system, we could fall into the pathological situation $\theta_\epsilon\equiv 0$.  
As it will be clear from the discussion, the approximate problem \eqref{amfgobs} admits smooth solutions even without additional elliptic regularization terms. 
This remarkable property is also true for certain first-order mean-field games, see, for instance, \cite{E1}. The function $u_\epsilon$ in 
\eqref{amfgobs} is the
value function for an optimal stopping problem.
This problem may not admit a continuous solution,  \cite{MR921827}, \cite{MR957658}.
Owing to 
the structure of \eqref{amfgobs}, we were able to 
prove regularity estimates that hold uniformly in $\epsilon$. However, in other related important situations, this may not be the case. 
It would be extremely interesting to consider a discontinuous viscosity solution 
approach for such problems. 

As we will show in Section \ref{convsec}, by passing to the limit in \eqref{amfgobs}, we obtain the mean-field obstacle problem 
\begin{equation}\label{obstacleepmfglimiteqrem}
\begin{cases}
H(Du,x) = g(\theta)&\text{in }\quad \Tt^N,\\
-\di (D_p H(Du,x)\theta)\leq \gamma(x)\quad&\text{in } \quad \Tt^N,\\
-\di (D_p H(Du,x)\theta)= \gamma(x)&\text{in } \quad \{u<\psi\}.\\
u\leq \psi.
\end{cases}
\end{equation}

This paper is structured as follows:
after discussing the main hypothesis in Section \ref{opmfg}, we prove, in Section \ref{A priori estimates}, various estimates for \eqref{amfgobs} that are 
 uniform in $\epsilon$. Namely, we obtain:
\begin{thm}\label{lipestimthm}
Under the assumptions of Section \ref{opmfg}, let $(u_\epsilon, \theta_\epsilon)$ be the solution to \eqref{amfgobs}.

Then, 
there exists a constant $C$ independent of $\ep$ such that 
\beq\label{w22est}\|u_\epsilon\|_{W^{2, 2}(\Tt^N)}\le C, \eeq
\beq\label{thetabounded} \|\theta_\epsilon\|_\infty\le C,\eeq
\beq\label{thetagradest}\|\theta_\epsilon\|_{W^{1, 2}(\Tt^N)}\le C,\eeq
and \beq\label{Dubounded} \|Du_\epsilon\|_\infty\le C.\eeq
\end{thm}
Applying these estimates, we consider the limit $\epsilon\to 0$
in Section \ref{convsec}. There we
get the following result:
\begin{thm}
\label{cthm}
Under the assumptions of Section \ref{opmfg}, 
let $(u_\epsilon, \theta_\epsilon)$ be the solution to \eqref{amfgobs}. Then there exists
$u\in W^{1, \infty}(\Tt^d)\cap W^{2, 2}(\Tt^d)$, and $\theta\in L^\infty(\Tt^d) \cap W^{1, 2}(\Tt^d)$ such that, through some subsequence, 
\beqs u_\ep\rightarrow u\quad\text{in }L^\infty(\Tt^N),\eeqs
\beqs Du_\ep\rightarrow Du,\quad  \theta_\ep\rightarrow\theta \quad\text{in }L^2(\Tt^N),\eeqs 
\beqs D^2u_\ep\rightharpoonup D^2u \quad\text{in }L^2(\Tt^N),\eeqs
as $\epsilon\rightarrow 0$. 
Furthermore, $(u, \theta)$ solves \eqref{obstacleepmfglimiteqrem}.
\end{thm}

Finally, in Section \ref{uniqsec} we establish the uniqueness of solution of the limit problem. More precisely, 
our main result is: 
\begin{thm}
\label{uthm}
Under the assumptions of Section \ref{opmfg}, 
there exists a unique solution $(u, \theta)$
$u\in W^{1, \infty}(\Tt^d)\cap W^{2, 2}(\Tt^d)$ and $\theta\in L^\infty(\Tt^d) \cap W^{1, 2}(\Tt^d)$
of the mean-field obstacle problem \eqref{obstacleepmfglimiteqrem}.
\end{thm}
 


\section{Assumptions}
\label{opmfg}

In this section, we describe our main assumptions. First, to 
ease the presentation, we assume the obstacle to vanish, that is, $\psi\equiv 0$. This entails no loss
of generality as we can always redefine the Hamiltonian and the solution so that the
new obstacle
 vanishes. In addition, 
we will take the source term $\gamma(x)=1$. However, our results can be easily adapted to deal with
a non-vanishing smooth source $\gamma$. 

%

On the Hamiltonian $H$  and the function $g$ we assume:
\begin{itemize}
\item [(i)] $H:\R^N\times\R^N\rightarrow\R$ is smooth and positive;
\item [(ii)]For each $p\in\R^N$, $x\rightarrow H(p,x)$ is periodic;
\item [(iii)] There exists a constant  $\lambda>0$ such that 
\beq\label{Hconvexity} H_{p_ip_j}(p,x)\xi_i\xi_j\ge \lambda|\xi |^2\eeq
for all $p,\,x,\,\xi\in\R^N$;
 \item [(iv)] There exists  $C>0$ such that 
 \beq\label{growthH}\begin{split}
 &|D^2_{pp}H|\le C\\ &|D^2_{xp}H|\le C(1+|p|)\\& |D^2_{xx}H|\le C(1+|p|^2)\end{split}\eeq
 and 
  \beq\label{dphp}H(p,x)-D_pH(p,x)p\le C\eeq
for all $p,\,x\in\R^N$.
\item [(vi)] $g:\R^+\rightarrow\R$ is smooth and such that
\begin{enumerate}
\item[(a)]
$g'>0$, 
\item[(b)]
$g^{-1}(0)>0$,  
\item[(c)]  $\theta\rightarrow \theta g(\theta)$ is convex, 
\item[(d)]
 there exist $C,\widetilde{C}>0$ and $\alpha\in[0,\alpha_0)$ with
$\alpha_0$ the solution of
\begin{equation}
\label{lipsalphaassmp}
2\alpha_0=(\alpha_0+1)\beta (\beta-1), \qquad \beta=\sqrt{\frac{2^*}{2}},
\end{equation}
if $N>2$, and $\alpha_0=\infty$ if $N\leq 2$, 
%
such that 
\beq\label{g'prop}C\theta^{\al-1}\leq g'(\theta)\leq \widetilde{C}\theta^{\al-1}+\widetilde{C},\eeq
\item[(f)]
 for any $C_0>0$ there exists $ C_1>0$ such that 
\beq\label{ggrowthprop}
C_0 \theta\leq \frac 1 2 g(\theta)\theta + C_1,
\eeq
for any $\theta\ge 0$.
\end{enumerate}

\end{itemize}



We choose a penalization term $\beta_\ep:\R\rightarrow\R$, smooth, with 
$0\le \beta_\ep'\le \frac{1}{\ep}$, $\beta_\ep''\geq 0$ and such that 
\beq\label{beta1}\beta_\ep(s)=0\quad \text{for } s\le 0,\quad\beta_\ep(s)=\frac{s-\ep}{\ep}\quad \text{for } s>2\ep\eeq
\beq\label{beta2} |\beta_\ep(s)-s\beta_\ep'(s)|\leq C\quad \text{for }s\in\R.\eeq

 \begin{rem}
The typical examples we have in mind for $g$ are
$$g(\theta)=\log(\theta),$$ and 
$$g(\theta)=\theta^\al+\theta_0,$$ for some $\theta_0>0$, $\al\in (0,\alpha_0)$ 
with $\alpha_0$ as in Assumption \ref{lipsalphaassmp}.
\end{rem}
\begin{rem}
The assumptions on the Hamiltonian imply that 
\beq\label{Hquadratic}\frac{\gamma}{2}|p|^2 -C\le H(p,x)\le C |p|^2 +C\eeq and 
\beq\label{DpHsulinear}\begin{split}& |D_pH(p,x)|\le C(1+|p|)\\&
|D_xH(p,x)|\le C(1+|p|^2)\end{split}
\eeq
 for all $p,\,x\in\R^N$.

\end{rem}

\section{A-priori estimates}
\label{A priori estimates}

In this section, we will establish various a-priori estimates for smooth solutions of the approximate mean-field
obstacle problem. 
Because these estimates will be uniform in $\epsilon$,  we can pass to an appropriate limit 
as $\epsilon\to 0$, as explained in the next section. 

In what follows, we denote by $(u,\theta)$ a classical solution of
\eqref{amfgobs}, and we will omit the subscript $\epsilon$ for convenience.  

\begin{lem}\label{boundlem1} 
Under the assumptions of Section \ref{opmfg}, 
there exist  constants $C$, $\theta_0>0$ independent of $\ep$ such that
for any solution   $(u,\theta)$  of \eqref{amfgobs},
\beq\label{lowerboundtheta}\theta\geq \theta_0\quad \text{in }\Tt^N,\eeq
\beq\label{Hsbound}\intor \theta dx\le C,\eeq 
\beq\label{Hsbound2}\left|\intor \theta g(\theta)dx\right|\le C,\eeq
\beq\label{uL1bound} \intor |u| dx\le C,\eeq and
\beq\label{duthetaboundlem2} \intor |Du|^2\theta dx\le C.\eeq 
\end{lem}
\dim
The lower bound on $\theta$ is a consequence of the fact that $g^{-1}$ is increasing with $g^{-1}(0)>0$, and $H$ and $\beta_\ep$ are non-negative:
$$\theta=g^{-1}(H(Du,x)+\beta_\ep(u))\ge g^{-1}(0)=:\theta_0>0.$$
Next, multiplying  the first equation of \eqref{amfgobs} by $\theta$, the second equation by $u$, integrating and subtracting, we get
\beqs\begin{split} \intor g(\theta)\theta dx=&\intor (H(Du,x)+\beta_\ep(u))\theta dx\\&
=\intor (H(Du,x)-D_pH(Du,x)Du)\theta dx \\&+\intor(\beta_\ep(u)-\beta_\ep '(u)u)\theta dx +\intor u dx.
\end{split} 
\eeqs
Then, using \eqref{dphp} and  \eqref{beta2},  we can find a constant $C_0>0$ such that
\beq
\label{lem1comp}  \intor g(\theta)\theta dx \le C_0 \intor \theta dx +\intor u dx \leq C_0 \intor \theta dx+\intor u^+ dx.
\eeq
Since, $g$ satisfies \eqref{ggrowthprop},  
we deduce that 
\beq
\label{aaa}
\frac{1}{2}\intor g(\theta)\theta dx \le \intor u^+ dx+C_1.
\eeq 
Since $H\geq 0$, $\beta_\epsilon(u)\leq g(\theta)$. 
In particular, \eqref{aaa} implies 
\beqs \intor \beta_\ep(u)dx \le \intor g(\theta)dx \le \frac{1}{\theta_0}\intor g(\theta)\theta dx\le C\intor u^+ dx+C.\eeqs
At the same time, by \eqref{beta2}, 
\beqs  \intor \beta_\ep(u)dx\ge  \int_{\{u>2\ep\}} \beta'_\ep(u)udx-C=\frac{1}{\ep} \int_{\{u>2\ep\}} u dx-C.\eeqs
Hence 
\beqs \frac{1}{\ep}  \intor u^+dx \le C\intor u^+ dx+C,\eeqs from which, for $\ep$ small enough, we get 
\beq \label{uplus}\intor u^+ dx\le C\ep.\eeq

We infer,   in particular,  that $\intor g(\theta)\theta dx \le C$  from which \eqref{Hsbound} follows. On the other hand,  the convexity of $\theta g(\theta)$ implies 
 \beqs \intor \theta g(\theta)dx \geq \left( \intor \theta dx \right) g\left(\intor \theta dx\right)\geq -C\eeqs and 
 \eqref{Hsbound2} is then proven. 
 
 Estimate \eqref{uL1bound} can be proven observing that \eqref{lem1comp} combined with \eqref{Hsbound}, \eqref{Hsbound2}, and 
 estimate \eqref{uplus}
 yields 
 \[
\left| \intor u dx\right|\leq C. 
 \]
 This estimate, combined with \eqref{uplus} implies, 
 \[
 \intor u^- dx \leq C
 \]
 from which then  \eqref{uL1bound} follows. 
 

Finally, using the first equation of \eqref{amfgobs} and  \eqref{Hquadratic} we get
 \beqs \intor |Du|^2\theta dx \le  C\intor g(\theta)\theta dx+C\intor\theta dx\eeqs 
 and then \eqref{duthetaboundlem2} is a consequence of \eqref{Hsbound} and \eqref{Hsbound2}.
\finedim

\begin{lem}\label{boundlem2} 
Under the assumptions of Section \ref{opmfg}, 
there exists a constant $C>0$ independent of $\ep$ such that
for any solution   $(u,\theta)$  of \eqref{amfgobs}
\beq\label{H2estim}\|u\|_{W^{2,2}(\Tt^N)}\leq C,\eeq 
\beq\label{thetagbound} \intor g'(\theta)|D\theta|^2dx\leq C\eeq  
and 
\beq\label{thetagbound2}\|\theta^\frac{\al+1}{2}\|_{W^{1,2}(\Tt^N)}\leq C.\eeq
\end{lem}
\dim
Using \eqref{Hquadratic}, \eqref{lowerboundtheta} and  \eqref{Hsbound2} , we get 
\begin{align*}
& \intor |Du|^2dx\leq C \intor H(Du,x)dx+C\le C\intor g(\theta)dx+C\\&\quad \le C\intor \theta g(\theta)dx+C\le C.
\end{align*}
The previous bound on $\intor |Du|^2dx$, estimate \eqref{uL1bound},  and the Poincar\'{e} inequality imply
\beqs \|u\|_{L^{2}(\Tt^N)}\le C.\eeqs 

Next, differentiating twice with respect to $x_i$  the first equation in  \eqref{amfgobs}, and then summing on $i$ (we use Einstein's convention, that is,
summing over repeated indices) we get
\beqs\begin{split}D_p H\cdot D(\Delta u)&+\Delta_x  H+ 2 H_{x_ip_j}u_{x_jx_i}+H_{p_jp_l}u_{x_jx_i}u_{x_lx_i}\\&
+\beta_\ep'(u)\Delta u+\beta_\ep''(u)|Du|^2=\Delta(g(\theta)).
\end{split}\eeqs

Multiplying the previous equation by $\theta$ and using that 
\begin{align*} &\intor (D_p H\cdot D(\Delta u)+\beta_\ep'(u)\Delta u)\theta dx=\intor (-\di( D_p H\theta)+\beta_\ep'(u)\theta)\Delta u dx\\&\quad =\intor \Delta u dx=0,
\end{align*}
we obtain
\begin{align*}
&\intor(\Delta_x  H+ 2 H_{x_ip_j}u_{x_jx_i}+H_{p_jp_l}u_{x_jx_i}u_{x_lx_i})\theta dx\\&\quad=
-\intor \beta_\ep''(u)|Du|^2\theta dx+\intor \Delta(g(\theta))\theta dx.
\end{align*}
The uniformly convexity of $H$, properties \eqref{growthH}, the convexity of $\beta_\ep$ and \eqref{duthetaboundlem2} then imply
\beqs C\intor |D^2 u|^2\theta dx+ \intor g'(\theta)|D\theta|^2dx \leq C \intor |D u|^2\theta dx+C\leq C\eeqs which gives, in particular, \eqref{thetagbound}. Moreover, 
from the previous inequality  and \eqref{lowerboundtheta}, we infer that 
\beqs \intor |D^2 u|^2dx\le C. 
\eeqs 
   This concludes the proof of \eqref{H2estim}.
   Finally, from \eqref{thetagbound} and \eqref{g'prop} we infer that 
    \beqs \intor\theta^{\al-1}|D\theta|^2dx\leq C  \intor g'(\theta)|D\theta|^2dx\leq C,\eeqs
   that is $|D\theta^\frac{\al+1}{2}|\in L^2(\Tt^N)$. Since, in addition $\theta\in L^1(\Tt^N)$, the previous estimate and the Poincar\'{e} inequality imply that 
  $\theta^\frac{\al+1}{2}\in L^2(\Tt^N)$ and so \eqref{thetagbound2} holds.
  \finedim 
We end this section with the proof of Theorem \ref{lipestimthm}. 

\begin{proof}[Proof of Theorem \ref{lipestimthm}]


Estimate \eqref{w22est} follows from lemma \ref{boundlem2}. Therefore, we proceed to prove the remaining bounds.
 First, we remark that  from assumption \eqref{g'prop}, $g$ satisfies 
\beq\label{gbehavioral>0}g(\theta)\leq C\theta^{\al}+C,\quad\text{when }\al>0,\eeq
and 
\beq\label{gbehavioral=0} g(\theta)\leq C\log(\theta)+C,\quad\text{when }\al=0.\eeq
Next, we show that $\beta'_\ep(u)$ is bounded uniformly in $\ep$. The function $s\rightarrow \beta'_\ep(s)$ is increasing. Hence $\beta'_\ep(u)$ attains its maximum 
where $u$ has the maximum. Let $x_0$ be a maximum point of $u$, then $Du(x_0)=0$ and $D^2u(x_0)\leq 0$ and from \eqref{amfgobs}, at $x=x_0$ we have
\beqs
\begin{split}
1&=-\di (D_p H(Du,x)\theta)+\beta_{\ep}'(u)\theta\\&
=-H_{p_ip_j}u_{x_ix_j}\theta-H_{p_ix_i}\theta-\frac{1}{g'(\theta)}(u_{x_ix_j}H_{p_i}H_{p_j}+H_{x_i}H_{p_i}+\beta_\ep'(u)u_{x_i}H_{p_i})+\beta_\ep'(u)\theta\\&
\ge -H_{p_ix_i}\theta-\frac{H_{x_i}H_{p_i}}{g'(\theta)}+\beta_\ep'(u)\theta.
\end{split}
\eeqs
Then

\beqs \max \beta_\ep'(u)=\beta_\ep'(u(x_0))\leq H_{p_ix_i}(0,x_0)+\frac{H_{x_i}(0,x_0)H_{p_i}(0,x_0)}{g'(\theta(x_0))\theta(x_0)}+\frac{1}{\theta(x_0)}.\eeqs

Using the properties of the Hamiltonian, \eqref{g'prop} and \eqref{lowerboundtheta}, we conclude that 
\beq\label{beta'bounded} \max \beta_\ep'(u)\le C.\eeq 
Next, we claim that, for some constant $C$ independent on $p$,  
\beq\label{thetagradpropthmclaim1} \intor\theta^{p-1}|D\theta|^2dx\leq C\intor\theta^{p+1}(1+|Du|^4)dx.\eeq
In order to prove \eqref{thetagradpropthmclaim1}, we use the technique from  \cite{E1}  (see the proof of Theorem 5.1) and multiply equation the second equation in \eqref{amfgobs} by $\di (\theta^p D_pH(Du,x))$, for $p>0$, and integrate by parts:
\beq\label{lemlipine1}\begin{split} \intor \beta_\ep'\theta\di (\theta^p D_pH)dx&=\intor (\theta H_{p_i})_{x_i}  (\theta^p H_{p_j})_{x_j}dx \\&
=\intor (\theta H_{p_i})_{x_j}  (\theta^p H_{p_j})_{x_i}dx \\&
=\intor (\theta (H_{p_i})_{x_j}+\theta_{x_j}H_{p_i})(\theta^p (H_{p_j})_{x_i}+p\theta^{p-1}\theta_{x_i}H_{p_j})dx\\&
=\intor \theta^{p+1}(H_{p_i})_{x_j}(H_{p_j})_{x_i}+p\theta^{p-1}H_{p_i}\theta_{x_i}H_{p_j}\theta_{x_j}\\&+(p+1)\theta^p\theta_{x_i}H_{p_j}(H_{p_i})_{x_j}dx\\&
=:\intor I_1+I_2+I_3 dx.
\end{split}\eeq
Using  assumptions \eqref{growthH} on $H$, we get
\beqs\begin{split}
I_1&=\theta^{p+1}(H_{p_ip_k}u_{x_kx_j}+H_{p_ix_j})(H_{p_jp_l}u_{x_lx_i}+H_{p_jx_i})\\&
\ge \theta^{p+1}[\gamma^2|D^2u|^2-C(1+|Du|)|D^2u|-C(1+|Du|^2)]\\&
\ge  \theta^{p+1}\tilde{\gamma}^2|D^2u|^2 -C \theta^{p+1}(1+|Du|^2),
\end{split}
\eeqs for some $\tilde{\gamma}>0$. 
Clearly $$I_2=p\theta^{p-1}|D_pH\cdot D\theta|^2.$$ Let us estimate $I_3$ from below. 
From the first equation of \eqref{amfgobs}, we gather that
\beq\label{gradthetasunstitution}H_{p_j}u_{x_jx_l}=g'(\theta)\theta_{x_l}-H_{x_l}-\beta_\ep'u_{x_l}.\eeq Assumption \eqref{g'prop} and the lower bound on $\theta$
\eqref{lowerboundtheta}, imply the existence of a positive constant $C_0$ such that
\beq\label{g'thetathetalower}g'(\theta)\theta\ge C_0>0.\eeq 
 Then, using the properties of the Hamiltonian, \eqref{beta'bounded} , \eqref{gradthetasunstitution} and \eqref{g'thetathetalower}, we get
\beqs\begin{split}I_3&=(p+1)\theta^p\theta_{x_i}H_{p_j}(H_{p_ip_l}u_{x_lx_j}+H_{p_ix_j})\\&
=(p+1)g'(\theta)\theta^{p}H_{p_ip_l}\theta_{x_i}\theta_{x_l}+(p+1)\theta^p\theta_{x_i}(H_{p_j}H_{p_ix_j}-H_{p_ip_l}H_{x_l})-(p+1)\theta^p\beta_\ep'H_{p_ip_l}\theta_{x_i}u_{x_l}\\&
\ge (p+1)\gamma C_0\theta^{p-1}|D\theta|^2-C(p+1)\theta^p|D\theta|(1+|Du|^2)-C(p+1)\theta^p|D\theta||Du|\\&
\ge C(p+1)\theta^{p-1}|D\theta|^2-C(p+1)\theta^{p+1}(1+|Du|^4).
\end{split}
\eeqs

Next, let us bound from above the left-hand side of \eqref{lemlipine1}. We have
\beqs\begin{split} \beta_\ep'\theta\di (\theta^p D_p H)&= \beta_\ep'\theta(p\theta^{p-1}D_pH \cdot D\theta+\theta^pH_{p_ip_j}u_{x_jx_i}+\theta^pH_{p_ix_i})\\&
\le p \theta^{p-1}|D_pH \cdot D\theta|^2+\tilde{\gamma}^2\theta^{p+1}|D^2u|^2+Cp\theta^{p+1}(1+|Du|),
\end{split}
\eeqs
where, again, we used the properties of the Hamiltonian and \eqref{beta'bounded}. 


From the preceding estimates, we conclude that 
\beqs\begin{split}
C (p+1)\intor\theta^{p-1}|D\theta|^2dx&-C(p+1)\intor\theta^{p+1}(1+|Du|^4)dx+p\intor\theta^{p-1}|D_pH\cdot D\theta|^2 dx\\&
+\tilde{\gamma}^2\intor \theta^{p+1}|D^2u|^2 -C\intor \theta^{p+1}(1+|Du|^2)dx\\&
\le \intor I_1+I_2+I_3dx\\&
=\intor \beta_\ep'\theta\di (\theta^p D_pH)dx\\&
\leq  p\intor \theta^{p-1}|D_pH\cdot D\theta|^2dx+\tilde{\gamma}^2\intor\theta^{p+1}|D^2u|^2dx\\&+Cp\intor\theta^{p+1}(1+|Du|)dx.
\end{split}
\eeqs 
The previous inequalities imply \eqref{thetagradpropthmclaim1}.

By Lemma \ref{boundlem2}, 
if $N>2$, we have $\theta\in L^{\frac{2^* (1+\alpha)}{2}}$, and for
$N=2$, $\theta\in L^p$, for all $p$. If $N=1$, the \eqref{thetabounded} holds trivially by
Morrey's theorem. 
	
Assume $N>2$, then Sobolev's inequality provides the bound
\beq\label{sobolinthmlip}\begin{split} \left(\intor \theta^{\frac{p+1}{2}2^*}dx\right)^\frac{2}{2^*}&
\leq C\intor \theta^{p+1}dx+C\intor|D(\theta^{\frac{p+1}{2}})|^2dx\\&
= C\intor \theta^{p+1}dx+C(p+1)^2\intor\theta^{p-1}|D\theta|^2dx.
\end{split}
\eeq
Let $\beta:=\sqrt{\frac{2^*}{2}}=\sqrt{\frac{N}{N-2}}>1$, then assumption \eqref{lipsalphaassmp} can be rewritten in the following way 
$$2\al\leq (\alpha+1)\beta^2\frac{\beta-1}{\beta}$$
and,  for $\al>0$, it implies, together with \eqref{Hquadratic} and  \eqref{gbehavioral>0}
that  
\beqs|Du|^4\leq C(g(\theta))^2+C\leq C\theta^{2\al}+C\leq C(1+\theta^{(\alpha+1)\beta^2\frac{\beta-1}{\beta}}).\eeqs 
The same inequality holds when  $\al=0$, using  \eqref{gbehavioral=0}:
 \beqs|Du|^4\leq C(g(\theta))^2+C\leq C(\log(\theta))^2+C\leq C(1+\theta^{(\alpha+1)\beta^2\frac{\beta-1}{\beta}}).\eeqs
 Therefore, from H{\"o}lder inequality we get
 \beqs\begin{split}\intor \theta^{p+1}(1+|Du|^4)dx&\leq C\intor \theta^{p+1}(1+\theta^{(\alpha+1)\beta^2\frac{\beta-1}{\beta}})dx\\&
 \leq C \intor \theta^{p+1}dx+C\left(\intor \theta^{(p+1)\beta}dx\right)^\frac{1}{\beta}\left(\intor \theta^{(\alpha+1)\beta^2} dx\right)^\frac{\beta-1}{\beta}\\&
 \leq C \intor \theta^{p+1}dx+C\left(\intor \theta^{(p+1)\beta}dx\right)^\frac{1}{\beta}\\&
 \leq C\left(\intor \theta^{(p+1)\beta}dx\right)^\frac{1}{\beta}.
\end{split} \eeqs
The last inequality, \eqref{thetagradpropthmclaim1} and \eqref{sobolinthmlip} give the estimate
\beqs\left(\intor\theta^{(p+1)\beta^2}dx\right)^\frac{1}{\beta^2}\leq Cp^2 \left(\intor\theta^{(p+1)\beta}dx\right)^\frac{1}{\beta}.
\eeqs
Arguing as in \cite{E1}, we get \eqref{thetabounded} and hence  \eqref{Dubounded} for $N>2$. 

When $N\le 2$, the reasoning is similar, because $\theta\in L^p$ for any $p$, 
and so \eqref{thetabounded} holds too. 

Finally, \eqref{thetagradest} is a consequence of 
\eqref{thetagbound2} and the estimate \eqref{thetabounded} just proven.


\end{proof}


\section{Convergence}
\label{convsec}

In this section, we present the proof of Theorem \ref{cthm} using the previous estimates. 

\begin{proof}[Proof of Theorem \ref{cthm}]
Let $(u_\ep,\theta_\ep)$ be a solution of \eqref{amfgobs}. The estimates obtained in the previous section, 
namely in Theorem \ref{lipestimthm}, imply the existence of  functions $u\in W^{2,2}(\Tt^N)\cap W^{1,\infty}(\Tt^N)$ and $\theta\in W^{1,2}(\Tt^N)\cap L^{\infty}(\Tt^N)$ such that, up to subsequence, 
as $\ep\rightarrow0$
\beqs u_\ep\rightarrow u\quad\text{in }L^\infty(\Tt^N),\eeqs
\beqs Du_\ep\rightarrow Du,\quad  \theta_\ep\rightarrow\theta \quad\text{in }L^2(\Tt^N),\eeqs 
\beqs D^2u_\ep\rightharpoonup D^2u \quad\text{in }L^2(\Tt^N).\eeqs
Furthermore, 
the sequence $u_\ep$ converges uniformly to a non-positive function $u$.

For $s>0$, since $\beta_\ep'$ is increasing ($\beta_\ep''>0$) and $\beta_\ep(0)=0$ we have 
$$\beta_\ep(s)=\beta_\ep(0)+\beta_\ep'(\xi_s)s\leq \max_{t\in[0,s]}\beta_\ep' (t)s=\beta_\ep'(s)s.$$ Therefore, for any $s\in\R$
$$\beta_\ep(s)\leq\beta_\ep'(s)s^+.$$ 

This implies, using \eqref{beta'bounded}
$$
0\leq \beta_\ep(u_\ep)\leq \beta_\ep'(u_\ep)(u_\ep)^+\leq C(u_\ep)^+.
$$
We conclude that $\beta_\ep(u_\ep)\rightarrow 0$ uniformly as $\ep\rightarrow 0$.
Hence, the limit $(u, \theta)$ solves
\[
H(Du,x) = g(\theta)\quad\text{in }\quad \Tt^N,
\]
\[
-\di (D_p H(Du,x)\theta)\leq 1\quad\text{in } \quad \Tt^N,
\]
\[
-\di (D_p H(Du,x)\theta)= 1\quad\text{in } \quad \{u<0\}.
\]
$$u\leq 0.$$
\end{proof}

\section{Uniqueness}
\label{uniqsec}

We end the paper with the proof of uniqueness of solutions to \eqref{obstacleepmfglimiteqrem}. This will be based upon a modified 
monotonicity argument inspired by the original technique by Lasry and Lions, see \cite{ll1, ll2, ll3}.  

\begin{proof}[Proof of Theorem \ref{uthm}]
Let $(u_1, \theta_1)$ and $(u_2, \theta_2)$ be distinct solutions of \eqref{obstacleepmfglimiteqrem}. Set 
$$A:=\{u_1-u_2>0\}.$$
$A$ is an open set. Moreover, $A\subset \{u_2<0\}$ since $u_1-u_2=u_1\leq 0$ in $\{u_2=0\}$, therefore, 
 \beqs-\text{div}(D_pH(Du_2,x)\theta_2)=\gamma(x)\quad\text{in }A.\eeqs
From the first equality in \eqref{obstacleepmfglimiteqrem}, we have 
\begin{equation*}\int_A [H(Du_1,x)-H(Du_2,x)](\theta_1-\theta_2)dx=\int_A(g(\theta_1)-g(\theta_2))(\theta_1-\theta_2)dx.\end{equation*}
Using  the second inequality for $u_1$ and the third equality for $u_2$ in \eqref{obstacleepmfglimiteqrem}, multiplying by  $ u_1-u_2>0$ in $A$ and integrating by parts, we obtain
\begin{equation*}\begin{split}0&\leq  \int_A \text{div}(D_p H(D u_1, x)\theta_1-D_pH(Du_2,x)\theta_2)(u_1-u_2)dx
\\&=-\int _A(D_p H(D u_1, x)\theta_1-D_pH(Du_2,x)\theta_2)D(u_1-u_2)dx.\end{split}\end{equation*}
Note that there is no boundary data since $u_1-u_2=0$ on $\partial A$. 
Adding the two inequalities and using the convexity of $H$, we get
\begin{equation*}\begin{split} 
0\leq \int_A(g(\theta_1)-g(\theta_2))(\theta_1-\theta_2)&\leq \int_A [H(Du_1,x)-H(Du_2,x)](\theta_1-\theta_2)dx\\&
-\int _A(D_p H(D u_1, x)\theta_1-D_pH(Du_2,x)\theta_2)D(u_1-u_2)dx\\&
=-\int _A[H(Du_2,x)-H(Du_1,x)-D_p H(D u_1, x)D(u_2-u_1)]\theta_1dx\\&
-\int _A[H(Du_1,x)-H(Du_2,x)-D_p H(D u_2, x)D(u_1-u_2)]\theta_2dx\\&
\leq -C\int _A |D(u_1-u_2)|^2dx.
\end{split}\end{equation*}
Thus we infer that $|A|=0$, i.e., $u_1\leq u_2$ almost everywhere. 
\end{proof}

\bibliographystyle{alpha}

\bibliography{mfg}

\end{document}